\documentclass[11pt]{article}
\usepackage{amssymb}
\usepackage{color}
\usepackage{appendix}
\usepackage[normalem]{ulem}

\newcommand{\bN}{\mathbf{N}}

\newcommand{\bK}{\mathbf{K}}
\newcommand{\ord}{\mbox{\rm ord}}

\newlength{\szer}

\newtheorem{defi}{Definition}[section]
\newtheorem{nota}[defi]{Remark}
\newtheorem{ejemplo}[defi]{Example}

\newtheorem{teorema}[defi]{Theorem}
\newtheorem{prop}[defi]{Proposition}
\newtheorem{lema}[defi]{Lemma}

\newenvironment{proof}[1][Proof]{\textbf{#1.} }{\
\rule{0.5em}{0.5em}}

\begin{document}
\title{Semigroups corresponding to branches at infinity of coordinate lines in the affine plane
\footnotetext{
     \noindent   \begin{minipage}[t]{4in}
       {\small
       2000 {\it Mathematics Subject Classification:\/} Primary 14R10;
       Secondary 32S05.\\
       Key words and phrases: branch at infinity, semigroup, cha\-rac\-teristic sequence, polynomial  automorphism, Abhyankar-Moh inequality.\\
       The first-named author was partially supported by the Spanish Project
       MTM2012-36917-C03-01.}
       \end{minipage}}}

\author{Evelia R.\ Garc\'{\i}a Barroso,  Janusz Gwo\'zdziewicz and Arkadiusz P\l oski}

\maketitle

\begin{abstract}
\noindent We characterize in terms of characteristic sequences the semigroups corresponding to branches at
infinity of plane affine curves $\Gamma$ for which there exists a polynomial automorphism mapping $\Gamma$ onto the axis $x=0$.
\end{abstract}

\section*{Introduction}
\noindent Let $\bK$ be an algebraically closed field of arbitrary characteristic and let $\gamma$, $\gamma',\cdots$ be plane algebroid branches centered at a point $O$ of an algebraic nonsingular surface defined over $\bK$. The semigroup
$G(\gamma)$ of the branch $\gamma$ is a subsemigroup of $\bN$ consisting of $0$ and all intersection numbers $i(\gamma,\gamma')$, where $\gamma'$ varies over all algebroid curves not having $\gamma$ as a component. We have $\min (G(\gamma)\backslash\{0\})=\ord \,\gamma$ (the order (multiplicity) of the branch $\gamma$).

\noindent The semigroups of plane branches can be characterized in terms of sequences of genera\-tors. A sequence of positive integers $(r_0,\ldots,r_h)$ is said to be a {\em characteristic sequence} if it satisfies the following two axioms:

\begin{enumerate}
\item [$(\mathbf1)$] Set $d_k=\gcd(r_0,\ldots,r_{k-1})$ for $1\leq k \leq h+1$. Then $d_k>d_{k+1}$ for $1\leq k \leq h$ and
$d_{h+1}=1$.
\item [$(\mathbf2)$] $d_kr_k<d_{k+1}r_{k+1}$ for $1\leq k <h$.
\end{enumerate}

\noindent We call $r_0$ the {\em initial term} of the characteristic sequence $(r_0,\ldots,r_h)$.

\noindent Let $G=r_0\bN+\cdots+r_h\bN$ be the semigroup generated by a characteristic sequence. Then 
$r_k=\min(G\backslash(r_0\bN+\cdots+r_{k-1}\bN))$ for $1\leq k \leq h$ which shows that $G$ and $r_0$ determine the sequence $(r_0,\ldots,r_h)$.

\medskip
\noindent {\bf Bresinsky-Angerm\"uller Semigroup Theorem}
\begin{enumerate}
\item Let $\gamma,\lambda$ be a pair of branches, where $\lambda$ is nonsingular. Let $n=i(\gamma,\lambda)<+\infty$. Then the semigroup $G(\gamma)$ of the branch $\gamma$ is gene\-ra\-ted by a characteristic sequence with initial term $n$.
\item Let $G\subseteq \bN$ be a semigroup gene\-ra\-ted by a characteristic sequence with initial term $n>0$. Then there exists a pair of branches $\gamma,\lambda$, where $\lambda$ is a nonsingular branch such that $i(\gamma,\lambda)=n$ and $G(\gamma)=G$.
\end{enumerate}

\noindent The above theorem was proved in \cite{Bresinsky} (for char$\,\bK=0$), \cite{Angermuller} and \cite{Garcia} (for arbitrary characteristic) for the transversal case: $i(\gamma, \lambda)=\ord \,\gamma$. A characteristic-blind proof of the theorem for arbitrary pairs $\gamma, \lambda$ with  $\lambda\neq \gamma $ nonsingular is given in \cite{GBP}.

\medskip

\noindent It will be convenient to regard
$\bK^2$ as the projective plane $\mathbf P\bK^2$ without the line at infinity $L$. Let $\Gamma\subset \bK^2$ be an affine irreducible curve. We say that $\Gamma$ {\em has one branch at infinity} if its projective closure $\overline{\Gamma}$ 
intersects $L$ at only one point  $O$ and $\overline{\Gamma}$ has only one branch centered at this point. 

\medskip

\noindent Let $\lambda$ be the branch of the line at infinity $L$ centered at $O$.

\medskip

\noindent By the Bresinsky-Angerm\"uller Theorem there exists a (unique) characte\-ristic sequence $(r_0,\ldots,r_h)$ generating $G(\gamma)$ with initial term $r_0=i(\gamma,\lambda)=\deg \Gamma$. We call $(r_0,\ldots,r_h)$ the {\em characteristic of} $\Gamma$ {\em at infinity}.

\medskip

\noindent The following result is of fundamental importance to studying the plane affine curves with one branch at infinity.

\medskip

\noindent {\bf Abhyankar-Moh inequality} \\ Assume that $\Gamma$ is an affine irreducible curve of degree greater than $1$ with one branch at infinity and let $(r_0,\ldots,r_h)$ be the characteristic of $\Gamma$ at infinity. Suppose that $\gcd(\deg \Gamma, \ord \,\gamma)\not\equiv 0$ (mod char$\,\bK$). Then 
\begin{enumerate}
\item[$(\mathbf3)$] $d_hr_h<r_0^2$.
\end{enumerate}

\medskip

\noindent The condition $\gcd(\deg \Gamma, \ord \,\gamma)\not\equiv 0$ (mod char$\,\bK$) is automatically satisfied for char$\,\bK=0$ and is essential if char$\,\bK\neq 0$.  In \cite{Abhyankar-Moh} the Abhyankar-Moh inequality is formulated in terms of Laurent-Puiseux parametrizations of the branch $\gamma$ (see also \cite{Kang}). For the formulation given above we refer the reader to \cite{Russell} and \cite{GBP}.

\medskip

\noindent {\bf  Conductor Formula}\\
\noindent Let $\Gamma$ be an affine irreducible curve of degree greater than $1$, rational, nonsingular with one branch at infinity. Let $(r_0,\ldots,r_h)$ be the characteristic of $\Gamma$ at infinity.  Then
\begin{enumerate}
\item [$(\mathbf4)$] $\displaystyle \sum_{k=1}^h \left( \frac{d_k}{d_{k+1}}-1\right)r_k=(r_0-1)^2$.
\end{enumerate}

\noindent The Conductor Formula is a corollary to the genus formula for a plane curve in terms of its singularities.
In \cite{Abhyankar-Moh}  it is formulated in terms of Laurent-Puiseux parametrizations of the branch at infinity. 

\medskip

\noindent The aim of this note is to characterize the semigroups of nonnegative integers generated by the sequences satisfying the properties $(\mathbf 1)$-$(\mathbf 4)$. Our main result is a counterpart of the Bresinsky-Angerm\"uller Semigroup Theorem. We will not impose any restriction on the characteristic of $\bK$. The above quoted results gave motivation for writing this paper but will be not used in our proofs.

\section{Result}

\noindent A sequence of positive integers $(r_0,\ldots,r_h)$ will be called an Abhyankar-Moh characteristic sequence if it has properties $(\mathbf 1)$-$(\mathbf 4)$ as in the Introduction. The following lemma is due to \cite{BGBP}.

\begin{lema} $\;$
\label{Bar}
\begin{enumerate}
\item [(i)] Let $(d_1,\ldots,d_{h+1})$ be a sequence of integers such that $d_1>\cdots > d_{h+1}=1$ and $d_{k+1}$ divides $d_k$ for $1\leq k\leq h$. Then the sequence $(r_0, r_1,\ldots, r_h)$ defined by $r_0=d_1$, 
$r_k=\frac{d_1^2}{d_k}-d_{k+1}$ for $1\leq k\leq h$ is an Abhyankar-Moh characteristic sequence with
$\gcd(r_0,\ldots,r_{k-1})=d_k$ for $1\leq k\leq h+1$.
\item [(ii)]  Let $(r_0, r_1,\ldots, r_h)$ be an Abhyankar-Moh characteristic sequence and let $d_k=\gcd(r_0,\ldots,r_{k-1})$ for $1\leq k\leq h+1$. Then $r_k=\frac{d_1^2}{d_k}-d_{k+1}$ for $1\leq k\leq h$.
\end{enumerate}
\end{lema}
\begin{proof} A simple calculation gives the proof of $(i)$. To check $(ii)$ let $q_k=\frac{n^2}{d_kd_{k+1}}-\frac{r_k}{d_{k+1}}$ for $1\leq k\leq h$. Then $q_k$ is an integer and $q_k=\frac{n^2-d_kr_k}{d_kd_{k+1}}\geq \frac{n^2-d_hr_h}{d_kd_{k+1}}> 0$. Hence $q_k\geq 1$ and $\frac{n^2}{d_k}-r_k=d_{k+1}q_k\geq d_{k+1}$, which implies 
\begin{equation}
\label{eee0}
\frac{n^2}{d_k}-d_{k+1}-r_k\geq 0 \;\;\hbox{\rm for } 1\leq k\leq h.
\end{equation}

\noindent On the other hand
\begin{eqnarray}
& &\;\;\; \sum_{k=1}^h\left(\frac{d_k}{d_{k+1}}-1\right)
\left(\frac{n^2}{d_k}-d_{k+1}-r_k\right) \nonumber \\
& & = \sum_{k=1}^h\left(\frac{d_k}{d_{k+1}}-1\right)\left(\frac{n^2}{d_k}-d_{k+1}\right)
-\sum_{k=1}^h\left(\frac{d_k}{d_{k+1}}-1\right)r_k \nonumber \\
& &=(n-1)^2-(n-1)^2=0. \label{eee1}
\end{eqnarray}

\noindent Combining $(\ref{eee0})$ and $(\ref{eee1})$ we get $r_k=\frac{n^2}{d_k}-d_{k+1}$ for $1\leq k\leq h$.
\end{proof}

\medskip

\noindent An affine curve $\Gamma\subset \bK^2$ will be called a {\em coordinate line in the affine plane} (in short: {\em coordinate line}) if there exists a polynomial automorphism $(f,g):\bK^2\longrightarrow \bK^2$ such that $f=0$ is the minimal equation of $\Gamma$.

\medskip

\noindent Every coordinate line  is an {\em embedded line} that is an affine curve biregular to an affine line $\bK$  but the converse is not true if char$\,\bK\neq 0$ (see (\cite{Nagata}). Embedded lines are nonsingular, rational, with one branch at infinity.

\begin{ejemplo}
Let $\Gamma$ be a graph of a polynomial in one variable of degree $n>1$. Then $\Gamma$ is the coordinate line. If $\gamma$ is the unique branch at infinity of $\Gamma$ then $G(\gamma)=n\bN+(n-1)\bN$.
\end{ejemplo}

\noindent The main result of this note is

\begin{teorema} $\;$
\label{main}
\begin{enumerate}
\item Let $\Gamma$ be a coordinate line of degree $n>1$ with the branch at infinity $\gamma$. Then $G(\gamma)$ is generated by an Abhyankar-Moh characteristic sequence with initial term $n$.
\item Let $G\subseteq \bN$ be a semigroup generated by an Abhyankar-Moh characteristic sequence with initial term $n>1$. Then there exists a coordinate line $\Gamma$ of degree $n$ with the branch at infinity $\gamma$ such that $G(\gamma)=G$.
\end{enumerate}
\end{teorema}

\noindent The proof of Theorem \ref{main} is given in Section \ref{tres} of this note.

\begin{nota} If \hbox{\rm char}$\,\bK=0$ then by the famous Abhyankar-Moh theorem every embedded line  is a coordinate line. Determining the semigroups $G(\gamma)$ corresponding to branches $\gamma$ of embedded lines remains an open question if \hbox{\rm char}$\,\bK\neq 0$.
\end{nota}

\begin{ejemplo} [Semigroup in Nagata's example \cite{Nagata}, p. 154]

\noindent Let $\bK$ be a field of characteristic $p>0$ and let $a>1$ be an integer
coprime with $p$. Consider the polynomials $x(t)=t^{p^2}$, $y(t)=t+t^{ap}$. 
Then for $f(x,y)=(y^p-x^a)^p-x$ and $g(x,y)=y-(y^p-x^a)^a$ we have 
$f(x(t),y(t))=0$ and $g(x(t),y(t))=t$ which shows that the affine curve 
$\Gamma$ with equation $f(x,y)=0$ is an embedded line. 

\medskip 
\noindent We compute the semigroup of the branch at infinity  $\gamma$ of $\Gamma$. Let us distinguish two cases:

\medskip\noindent
\textbf{I.} If $a<p$ then the Zariski closure of $\Gamma$ intersects  the line at infinity at  $P=(1:0:0)$. 
We have $r_0=\deg \Gamma=p^2$, $r_1=\ord_P \bar\Gamma=p(p-a)$.
Thus $d_1=p^2$, $d_2=\gcd(r_0,r_1)=p$ and $d_3=1$.  
Substituting these numbers to the conductor formula 
$$ \left(\frac{d_1}{d_2}-1\right)r_1+\left(\frac{d_2}{d_3}-1\right)r_2=(r_0-1)^2 $$
we get $r_2=p^3+p(a-1)-1$.\\ 

\noindent That is $G(\gamma)=p^2\bN+p(p-a)\bN+(p^3+p(a-1)-1)\bN$.

\medskip\noindent
\textbf{II.} If $a>p$ then the Zariski closure of $\Gamma$ intersects  the line at infinity at  $Q=(0:1:0)$.  We have $r_0=\deg \Gamma=ap$, $r_1=\ord_Q \bar\Gamma=p(a-p)$.
Thus $d_1=ap$, $d_2=\gcd(r_0,r_1)=p$ and $d_3=1$.  
Substituting these numbers to the conductor formula we get $r_2=a^2p+p(a-1)-1$.\\

\noindent That is  $G(\gamma)=ap\bN+p(a-p)\bN+(a^2p+p(a-1)-1)\bN$.\\

\noindent In both cases the semigroup $G(\gamma)$ satisfies properties $(\mathbf1)$, $(\mathbf2)$, $(\mathbf4)$ but not $(\mathbf3)$.
\end{ejemplo}

\section{Proof}
\label{tres}

\noindent The following lemma is well-known.

\begin{lema}
\label{l1}
Let $\gamma\neq \lambda$ be plane branches, where $\lambda$  is nonsingular. Let $n=i(\gamma,\lambda)$. Suppose that there exist a characteristic sequence $(r_0,\ldots,r_h)$ with initial term $r_0=n$ and a sequence of branches $(\gamma_1, \ldots, \gamma_{h+1})$, $\gamma_{h+1}=\gamma$ such that
\begin{enumerate}
\item [(1)] $i(\gamma_k,\lambda)=\frac{n}{d_k}$ for $1\leq k\leq h+1$,
\item [(2)] $i(\gamma_k,\gamma_{h+1})=r_k$ for $1\leq k\leq h$.
\end{enumerate}
\noindent Then $G(\gamma)=r_0\bN+\cdots+r_h\bN$.
\end{lema}
\begin{proof} 
See e.g. \cite{GBP}, Lemma 5.3. 
\end{proof}

\medskip

\noindent Let $\lambda$ be a nonsingular branch. For any branches $\gamma,\gamma'$ different from $\lambda$ we put

\[
d_{\lambda}(\gamma,\gamma')=\frac{i(\gamma,\gamma')}{i(\gamma,\lambda)i(\gamma',\lambda)}.
\]

\begin{lema}
\label{STI}
For any three branches $\gamma, \gamma',\gamma^{''}$ at least two of the numbers $d_{\lambda}(\gamma,\gamma'),d_{\lambda}(\gamma,\gamma^{''}),d_{\lambda}(\gamma',\gamma^{''})$ are equal and the third one is not smaller than the other two.
\end{lema}
\begin{proof} See \cite{GBP}, Theorem 2.2. 
\end{proof}

\begin{prop}
\label{pr1}
Let $(f_1,\ldots,f_{h+1})$ be a sequence of polynomials in $\bK[x,y]$ and let $(n_1,\ldots,n_h)$ be a sequence of integers greater than $1$ such that 
\begin{enumerate}
\item $1=\deg f_1<\ldots <\deg f_{h+1}$, 
\item $(f_k,f_{k+1}):\bK^2\longrightarrow \bK^2$ is a polynomial automorphism for $1\leq k\leq h$, 
\item $\deg f_{k+1}=n_k\deg f_k$ for $1\leq k\leq h$.
\end{enumerate}
\noindent Let $d_k=n_k\cdots n_h$ for $1\leq k\leq h$, $d_{h+1}=1$  and let $\Gamma$ be  the affine curve with minimal equation $f_{h+1}=0$, $\gamma$ its branch at infinity. Then 
$G(\gamma)=r_0\bN+\cdots+r_h\bN$, where $r_0=d_1$ and $r_k=\frac{d_1^2}{d_k}-d_{k+1}$ for $1\leq k\leq h$.
\end{prop}

\noindent \begin{proof}
Let $\Gamma_k\subseteq \bK^2$ be the affine curve with minimal equation $f_k=0$ and let $\gamma_k$ be the branch at infinity of $\Gamma_k$. In particular $\Gamma_{h+1}=\Gamma$ and $\gamma_{h+1}=\gamma$.
All branches $\gamma_k$, $1\leq k\leq h+1$ are centered at the common point at infinity $O$ of the curves $\Gamma_k$. Let $\lambda$ be the branch of the line at infinity $L$ centered at $O$. Let $n=i(\gamma,\lambda)$. Observe that $n=\deg \Gamma_{h+1}=n_1\cdots n_h=d_1$ and $i(\gamma_k,\lambda)=\deg \Gamma_k=n_1\cdots n_{k-1}=\frac{n}{d_k}$, that is the assumption $(1)$ of Lemma \ref{l1} is satisfied.

\medskip

\noindent Using B\'ezout's theorem to the curves $\overline{\Gamma}_k$, 
$\overline{\Gamma}_{k+1}$ which intersect in exactly one point in $\bK^2$ we get

\begin{equation}
\label{eq1}
i(\gamma_k,\gamma_{k+1})=\frac{n^2}{d_kd_{k+1}}-1
\end{equation}

\noindent since the intersection in $\bK^2$ is transversal. In particular $i(\gamma_h,\gamma_{h+1})=\frac{n^2}{d_hd_{h+1}}-1=\frac{n^2}{d_h}-d_{h+1}=r_h.$

\medskip

\noindent To check the assumption $(2)$ of Lemma \ref{l1} we proceed by descendent induction on $k$.

\noindent Assume that $i(\gamma_h,\gamma_{h+1})=r_h,\cdots, i(\gamma_{k+1},\gamma_{h+1})=r_{k+1}$. 
\noindent By inductive assumption  $d_{\lambda}(\gamma_{k+1},\gamma_{h+1})=1-\frac{d_{k+1}d_{k+2}}{d_1^2}$ and  by (\ref{eq1}) $d_{\lambda}(\gamma_{k+1},\gamma_{k})=1-\frac{d_{k+1}d_{k}}{d_1^2}$.

\noindent Let us consider three branches $\gamma_k, \gamma_{k+1},\gamma_h$. Since $d_{\lambda}(\gamma_{k+1},\gamma_{k})< d_{\lambda}(\gamma_{k+1},\gamma_{h+1})$ we get by Lemma \ref{STI} applied to 
$\gamma_k, \gamma_{k+1},\gamma_h$ that $d_{\lambda}(\gamma_{k},\gamma_{h+1})=d_{\lambda}(\gamma_{k+1},\gamma_{k})$ which implies $i(\gamma_k,\gamma_{h+1})=\frac{d_1^2}{d_k}\left(1-\frac{d_{k+1}d_k}{d_1^2}\right)=r_{k}$.
\end{proof}

\begin{prop}[Van der Kulk]
\label{vdk}
\noindent Let $(f,g):\bK^2\longrightarrow \bK^2$ be a polynomial automorphism. Then either $\deg f$ divides $\deg g$ or $\deg g$ divides $\deg f$.
\end{prop}
\noindent \begin{proof} See \cite{van der Kulk} or \cite{GBP}. \end{proof}

\begin{lema}
\label{coro}
Let $(g,f):\bK^2\longrightarrow \bK^2$ be a polynomial automorphism. If $\deg f>1$ then there exists 
$\tilde{g}$ such that $(\tilde{g},f):\bK^2\longrightarrow \bK^2$ is a polynomial automorphism and $\deg \tilde{g}<\deg f$.
\end{lema}

\noindent \begin{proof} 
\noindent If $\deg g< \deg f$ then we put $\tilde g=g$. Suppose that $\deg g\geq \deg f$. By Proposition \ref{vdk} $N=\frac{\deg g}{\deg f}$ is an integer. Each coordinate line has exactly one point at infinity. Since $(g,f)$ is a non-linear automorphism  the points at infinity of $g=0$ and $f=0$ coincide. Thus we can find a constant $c\in \bK$ such that $\deg (g-cf^N)<\deg g$ (cf. \cite{van der Kulk}, p. 37).
Replace $g$ by $g-cf^N$. Repeating this procedure a finite number of times we get a polynomial 
automorphism $(\tilde g,f):\bK^2\longrightarrow \bK^2$ such that $\deg \tilde g< \deg f$.
\end{proof}

\medskip
\noindent {\bf Proof of Theorem 2.1} \\
\noindent  {$(i)$} Let $\Gamma$ be a coordinate line  with the minimal equation $f=0$ of degree $n>1$. Let $\gamma$ be the branch at infinity of $\Gamma$.

\noindent Using Lemma \ref{coro} we construct a sequence of polynomials $(f_1, \ldots,f_{n+1})$, where $f_{h+1}=f$ such that $(f_k,f_{k+1}):\bK^2\longrightarrow \bK^2$ is a polynomial automorphism for $1\leq k\leq h$ and $\deg f_k< \deg f_{k+1}$. By Proposition \ref{vdk} $\deg f_k$ divides $\deg f_{k+1}$. Let $n_k=\frac{\deg f_{k+1}}{\deg f_k}$ for $1\leq k\leq h$.

\medskip

\noindent  Applying Proposition \ref{pr1} to the sequences $(f_1,\ldots,f_{h+1})$ and $(n_1,\ldots,n_h)$ we get 
$G(\gamma)=r_0\bN+\cdots+r_h\bN$, where $r_0=n$ and $r_k=\frac{n^2}{d_k}-d_{k+1}$ for $1\leq k\leq h$. The sequence $(r_0,\ldots,r_h)$ is an Abhyankar-Moh sequence by Lemma \ref{Bar} $(i)$.

\medskip

\noindent  {$(ii)$} Let $G\subseteq \bN$ be a semigroup generated by an Abhyankar-Moh sequence
$(r_0,\ldots,r_{h})$ with the initial term $r_0=n>1$. Let $d_k=\gcd(r_0,\ldots,r_{k-1})$ for $1\leq k\leq h+1$. Then $r_k=\frac{n^2}{d_k}-d_{k+1}$ by Lemma \ref{Bar} $(ii)$. Let
$n_k=\frac{d_k}{d_{k+1}}$ for $1\leq k\leq h+1$.

\noindent Set 
\[
\begin{array}{l}
f_1=y,\\
f_2=y^{n_1}-x,\\
f_{k+1}=f_k^{n_k}-f_{k-1}\;\;\hbox{\rm for } 2\leq k\leq h.
\end{array}
\]

\noindent Then the  sequences $(f_1, \ldots,f_{n+1})$ and $(n_1,\ldots,n_h)$ satisfy the assumptions of Proposition
 \ref{pr1} and it suffices to take $\Gamma$ as the plane affine curve with minimal equation $f_{h+1}=0$.

\medskip
\noindent
{\small Evelia Rosa Garc\'{\i}a Barroso\\
Departamento de Matem\'aticas, Estad\'{\i}stica e I.O.\\
Secci\'on de Matem\'aticas, Universidad de La Laguna\\
38271 La Laguna, Tenerife, Espa\~na\\
e-mail: ergarcia@ull.es}

\medskip

\noindent {\small  Janusz Gwo\'zdziewicz and Arkadiusz P\l oski\\
Department of Mathematics\\
Kielce University of Technology\\
Al. 1000 L PP7\\
25-314 Kielce, Poland\\
e-mail: matjg@tu.kielce.pl and matap@tu.kielce.pl}

\begin{thebibliography}{GGGGG}

\bibitem[A-M]{Abhyankar-Moh} Abhyankar, S.S.; Moh, T.T. Embeddings of the line in the plane. 
{\em J. reine angew. Math.} {\bf 276} (1975), 148-166.

\bibitem[Ang]{Angermuller} Angerm\"uller, G. Die Wertehalbgruppe einer ebener
irreduziblen algebroiden Kurve. Math. Z. {\bf 153} (1977), no. 3, 267-282.

\bibitem[B-GB-P] {BGBP} Barrolleta, R.D., Garc\'{\i}a Barroso, E.R., P\l oski, A. On the Abhyankar-Moh inequality.
arXiv:1407.0176.

\bibitem[Bre] {Bresinsky} Bresinsky, H. Semigroups corresponding to
algebroid branches in the plane. Proc. Amer. Math. Soc. \textbf{32}
(1972),  no. 2, 381-384.

\bibitem[GB-P]{GBP} Garc\'{\i}a Barroso, E. R., P\l oski, A. An approach to plane algebroid branches. arXiv:1208.0913. 4 Aug 2012, (accepted for publication in Revista Matem\'atica Complutense)

\bibitem[Gar-St] {Garcia} Garc\'{\i}a, A., St\"ohr, K.O. On semigroups of irreducible algebroid plane curves Comm. Algebra 15 (1987), no. 10, 2185-2192.

\bibitem[Kang] {Kang} Ming-Chang Kang. On Abhyankar Moh's epimorphism theorem. Amer. J. of Math. 113 (1991), 399-421.

\bibitem[Na]{Nagata} Nagata, M. A theorem of Gutwirth. J. Math. Kyoto Univ. 11 (1971), 149-154.

\bibitem[Ru]{Russell} Russell, P. Hamburger-Noether expansions and approximate roots of
polynomials. {\em Manuscripta Math.} {\bf 31} (1980), no. 1-3, 25-95.

\bibitem[vdK]{van der Kulk} van der Kulk, W. On polynomial rings in
two variables. {\em Nieuw Arch. Wiskunde (3)} 1, (1953) 33-41.
\end{thebibliography}
\end{document}